\documentclass[oneside,english]{amsart}
\usepackage[T1]{fontenc}
\usepackage[latin9]{inputenc}
\usepackage[a4paper]{geometry}
\usepackage{amsthm}
\usepackage{amstext}
\usepackage{amssymb}
\usepackage{esint}

\makeatletter
\numberwithin{equation}{section}
\numberwithin{figure}{section}
  \theoremstyle{plain}
  \newtheorem*{thm*}{\protect\theoremname}
\theoremstyle{plain}
\newtheorem{thm}{\protect\theoremname}
  \theoremstyle{plain}
  \newtheorem{lem}[thm]{\protect\lemmaname}
  \theoremstyle{plain}
  \newtheorem{prop}[thm]{\protect\propositionname}

\makeatother

\usepackage{babel}
  \providecommand{\lemmaname}{Lemma}
  \providecommand{\propositionname}{Proposition}
  \providecommand{\theoremname}{Theorem}
\providecommand{\theoremname}{Theorem}

\begin{document}

\title{A probabilistic interpretation of the Volkenborn integral}

\author{A. Bhandari$^{\left(1\right)}$ and C. Vignat$^{\left(2\right)}$}

\address{$\left(1\right)$ A.Bhandari was with the Biomedical Imaging Laboratory, E.P.F.L., Lausanne,
Switzerland, during the completion of this work $\left(2\right)$ Information Theory Laboratory, L.T.H.I.,
E.P.F.L., Lausanne, Switzerland}

\email{christophe.vignat@epfl.ch, ayush.bhandari@epfl.ch}
\begin{abstract}
In this paper, we provide a probabilistic interpretation of the Volkenborn
integral; this allows us to extend results by T. Kim et al about sums
of Euler numbers to sums of Bernoulli numbers. We also obtain a probabilistic
representation of the multidimensional Volkenborn integral which allows
us to derive a multivariate version of Raabe's multiplication theorem
for the higher-order Bernoulli and Euler polynomials.
\end{abstract}
\maketitle

\section{Introduction}

The Volkenborn integral was introduced in 1971 by A. Volkenborn in
his PhD dissertation and subsequently in the set of twin papers \cite{Volkenborn};
a more recent treatment of the subject can be found in \cite{Robert}.
The Volkenborn integral, or fermionic $p-$adic $q-$integral on $\mathbb{Z}_{p}$,
of a function $f$ is defined as 
\[
\int_{\mathbb{Z}_{p}}f\left(y\right)d\mu_{-q}\left(y\right)=\lim_{N\to+\infty}\frac{1+q}{1+q^{p^{N}}}\sum_{x=0}^{p^{N}-1}f\left(x\right)\left(-q\right)^{x}
\]

In particular, the $q=1$-Volkenborn integral satisfies \cite[eq. (1.6)]{Kim1}
\[
\int_{\mathbb{Z}_{p}}e^{\left(x+y\right)t}d\mu_{-1}\left(y\right)=\frac{2}{e^{t}+1}e^{xt}=\sum_{n=0}^{+\infty}E_{n}\left(x\right)\frac{t^{n}}{n!}.
\]
where $E_{n}\left(x\right)$ is the Euler polynomial of degree $n$.

Another interesting case is the $q=0$-Volkenborn integral which satisfies
\cite[p. 271]{Robert}
\[
\int_{\mathbb{Z}_{p}}e^{\left(x+y\right)t}d\mu_{0}\left(y\right)=\frac{te^{xt}}{e^{t}-1}=\sum_{n=0}^{+\infty}B_{n}\left(x\right)\frac{t^{n}}{n!}
\]
where $B_{n}\left(x\right)$ is the Bernoulli polynomial of degree
$n.$ From this result we deduce the following
\begin{thm*}
If $f\left(x\right)$ is analytic in a neighborhood of $0$ then its
$q=0$-Volkenborn integral can be computed as the expectation %
\footnote{in the rest of this paper, we use the notation $\mathbb{E}_{X}f\left(X\right)$
for the probabilistic expectation $\int f\left(x\right)p_{X}\left(x\right)dx$
where $ $$p_{X}$ is the probability density function of the random
variable $X.$ When no ambiguity occurs, we also denote $\mathbb{E}f\left(X_{1},X_{2},\dots\right)$
the expectation over all random variables that appear as arguments
of the function $f.$%
} 
\begin{equation}
\int_{\mathbb{Z}_{p}}f\left(x\right)d\mu_{0}\left(x\right)=\mathbb{E}f\left(x+\imath L_{B}-\frac{1}{2}\right)\label{eq:probabilistic}
\end{equation}
where the random variable $L_{B}$ follows the logistic distribution
with density
\begin{equation}
\frac{\pi}{2}sech^{2}\left(\pi x\right),\,\, x\in\mathbb{R}.\label{eq:logistic}
\end{equation}
Moreover, its $q=1$-Volkenborn integral coincides with the expectation
\[
\int_{\mathbb{Z}_{p}}f\left(x\right)d\mu_{1}\left(x\right)=\mathbb{E}f\left(x+\imath L_{E}-\frac{1}{2}\right)
\]
where the random variable $L_{E}$ follows the hyperbolic secant distribution
with density
\begin{equation}
sech\left(\pi x\right),\,\, x\in\mathbb{R}.\label{eq:logistic-1}
\end{equation}
\end{thm*}
\begin{proof}
These results can be proved by computing the characteristic functions
associated to the logistic distribution
\[
\mathbb{E}\exp\left(\imath tL_{B}\right)=\frac{t}{2}\text{csch}\left(\frac{t}{2}\right)
\]
and to the hyperbolic secant distribution
\[
\mathbb{E}\exp\left(\imath tL_{E}\right)=\text{sech}\left(\frac{t}{2}\right).
\]

\end{proof}
The special case $f\left(x\right)=x^{n}$ yields the following moment
representation for the Bernoulli polynomial of degree $n$ 
\begin{equation}
B_{n}\left(x\right)=\mathbb{E}\left(x+\imath L_{B}-\frac{1}{2}\right)^{n}\label{eq:Bn(x)}
\end{equation}
and 
\[
E_{n}\left(x\right)=\mathbb{E}\left(x+\imath L_{E}-\frac{1}{2}\right)^{n}
\]
for the Euler polynomial of degree $n.$ Moreover, choosing $x=0$
yields the following moment representation for the $n-$th Bernoulli
number
\[
B_{n}=B_{n}\left(0\right)=\mathbb{E}\left(\imath L_{B}-\frac{1}{2}\right)^{n}
\]
and 
\begin{equation}
E_{n}=E_{n}\left(0\right)=\mathbb{E}\left(\imath L_{E}-\frac{1}{2}\right)^{n}\label{eq:En}
\end{equation}
for the $n-$th Euler number %
\footnote{note that it differs from the $n-$th Euler number of the first kind
defined by $\tilde{E}_{n}=2^{n}E_{n}\left(\frac{1}{2}\right).$%
} where $L_{B}$ and $L_{E}$ follow respectively the logistic distribution
\eqref{eq:logistic} and the hyperbolic secant distribution \eqref{eq:logistic-1}.

An important feature of the logistic and hyperbolic secant random
variables is the following cancellation property:
\begin{lem}
If $U_{B}$ is a continuous random variable uniformly distributed
over $\left[0,1\right]$ and independent of $L_{B}$ then 
\[
\mathbb{E}\left(x+\imath L_{B}-\frac{1}{2}+U_{B}\right)^{n}=x^{n}.
\]
Accordingly, if $U_{E}$ is a Rademacher distributed random variable
($\Pr\left\{ U_{E}=0\right\} =\Pr\left\{ U_{E}=1\right\} =\frac{1}{2}$)
then 
\[
\mathbb{E}\left(x+\imath L_{E}-\frac{1}{2}+U_{E}\right)^{n}=x^{n}.
\]

\end{lem}
Both results can be shown considering the characteristic functions
of the involved variables; for example, in the Bernoulli case,
\[
\mathbb{E}\exp\left(t\left(\imath L_{B}-\frac{1}{2}\right)\right)\exp\left(tU_{B}\right)=1
\]
so that all integer nonzero moments of the random variable $\imath L_{B}-\frac{1}{2}+U_{B}$
vanish.

The Volkenborn integrals were used by Kim et al \cite{Kim1} to obtain
non-trivial identities on Euler numbers $E_{n}$ using integrals of
the Bernstein polynomials defined as
\[
\mathcal{B}_{k,n}\left(x\right)=\binom{n}{k}x^{k}\left(1-x\right)^{n-k},\,\,0\le x\le1.
\]

In this paper, we show that the probabilistic representation \eqref{eq:probabilistic}
of the Volkenborn integral makes its computation very easy to handle.
We illustrate this fact by extending the non-trivial identities of
\cite{Kim1} to the case of Bernoulli numbers. In the second section,
we derive the probabilistic equivalent of the multidimensional Volkenborn
integrals as introduced in \cite{Kim2} and we use it to prove a multivariate
version of Raabe's multiplication theorem for Bernoulli and Euler
polynomials.

\section{Identities for Bernoulli numbers and polynomials}

\subsection{first-order identity}

In order to obtain non-trivial identities on Bernoulli numbers, we
replace the Bernstein polynomials used by Kim et al by the Beta polynomials
\[
\mathfrak{B}_{k,n}\left(x\right)=x^{k}\left(1+x\right)^{n-k},\,\,0\le k\le n
\]
and, with $X=\imath L_{B}-\frac{1}{2},$ compute the expectation $\mathbb{E}\mathfrak{B}_{k,n}\left(X\right)$
in two different ways:

- the first way is by applying the binomial formula
\begin{equation}
\mathbb{E}\mathfrak{B}_{k,n}\left(X\right)=\mathbb{E}X^{k}\sum_{j=0}^{n-k}\binom{n-k}{j}X^{j}=\sum_{j=0}^{n-k}\binom{n-k}{j}B_{j+k}\label{eq:E1}
\end{equation}

- the second way is by expressing $X=\left(X+1\right)-1$ so that
\begin{equation}
\mathbb{E}\mathfrak{B}_{n,k}\left(X\right)=\mathbb{E}\sum_{j=0}^{k}\binom{k}{j}\left(-1\right)^{j}\left(1+X\right)^{\left(n-k\right)+\left(k-j\right)}=\sum_{j=0}^{k}\binom{k}{j}\left(-1\right)^{j}\left\{ B_{n-j}+\delta_{n-j-1}\right\} ;\label{eq:E2}
\end{equation}
since $j\le k\le n,$ the Kronecker adds a term $\binom{n-1}{n-1}\left(-1\right)^{n-1}$
if $k=n-1$ or $\binom{n}{n-1}\left(-1\right)^{n-1}$ if $k=n$ and
no term otherwise. We conclude the following
\begin{thm}
The Bernoulli numbers satisfy
\[
\sum_{j=0}^{n-k}\binom{n-k}{j}B_{j+k}=\begin{cases}
\sum_{j=0}^{k}\binom{k}{j}\left(-1\right)^{j}B_{n-j} & \text{if\,\,}0\le k\le n-2,\text{\,\,}n\ge2\\
\sum_{j=0}^{k}\binom{k}{j}\left(-1\right)^{j}B_{n-j}+\left(-1\right)^{n-1} & \text{if\,\,}k=n-1,\,\, n\ge1\\
\sum_{j=0}^{k}\binom{k}{j}\left(-1\right)^{j}B_{n-j}+n\left(-1\right)^{n-1} & \text{if\,\,}k=n,\,\, n\ge0
\end{cases}
\]

\end{thm}
We remark that the case $k=n$ reads
\[
B_{n}=\sum_{j=0}^{n}\binom{n}{j}\left(-1\right)^{j}B_{n-j}+n\left(-1\right)^{n-1}
\]

\subsection{polynomial identities}

From \eqref{eq:Bn(x)}, we deduce that the above results can be extended
to the case of Bernoulli polynomials by choosing $X=x+\imath L_{B}-\frac{1}{2}.$
We deduce the following
\begin{thm}
The Bernoulli polynomials satisfy, for all $0\le k\le n,$
\begin{equation}
\sum_{j=0}^{n-k}\binom{n-k}{j}B_{j+k}\left(x\right)=\sum_{j=0}^{k}\binom{k}{j}\left(-1\right)^{j}B_{n-j}\left(x\right)+\left(nx-\left(n-k\right)\right)x^{n-k-1}\left(x-1\right)^{k-1}\label{eq:Kim1}
\end{equation}
\end{thm}
\begin{proof}
The left-hand side is a direct consequence of that of \eqref{eq:E1};
the left-hand side of \eqref{eq:E2} with $X=x+\imath L_{B}-\frac{1}{2}$
yields
\[
\sum_{j=0}^{k}\binom{k}{j}\left(-1\right)^{j}B_{n-j}\left(x+1\right)
\]
and since $B_{n-j}\left(x+U\right)=x^{n-j},$ we deduce $B_{n-j}\left(x+1\right)-B_{n-j}\left(x\right)=\left(n-j\right)x^{n-j-1}.$
replacing in the above sum yields the result.
\end{proof}
We note that the case $k=0$ reads
\[
\sum_{j=0}^{n}\binom{n}{j}B_{j}\left(x\right)=B_{n}\left(x\right)+nx^{n-1},
\]
which can be restated as
\[
B_{n}\left(x+1\right)-B_{n}\left(x\right)=nx^{n-1}
\]
and is nothing but the expression of the cancellation principle
\[
\mathbb{E}B_{n-1}\left(x+U_{B}\right)=x^{n-1}.
\]

\section{A polynomial extension to Kim's identity}

In \cite{Kim1}, the following identity is derived using the Bernstein
polynomials $\mathcal{B}_{k,n}\left(x\right)=\binom{n}{k}x^{k}\left(1-x\right)^{n-k}$
\[
\sum_{j=0}^{n-k}\binom{n-k}{j}\left(-1\right)^{j}E_{j+k}=\sum_{j=0}^{k}\binom{k}{j}\left(-1\right)^{k-j}E_{n-j}+2\delta_{k}.
\]
We now provide the following polynomial extension of this identity
\begin{thm}
\label{thm:Kim_poly}The Euler polynomials satisfy
\[
\sum_{j=0}^{n-k}\binom{n-k}{j}\left(-1\right)^{j}E_{j+k}\left(x\right)=\left(-1\right)^{n+k+1}\sum_{j=0}^{k}\binom{k}{j}E_{n-j}\left(x\right)+2x^{k}\left(1-x\right)^{n-k}.
\]
\end{thm}
\begin{proof}
We start from the identity
\[
\sum_{j=0}^{n-k}\binom{n-k}{j}\left(-1\right)^{j}x^{j+k}=\sum_{j=0}^{k}\binom{k}{j}\left(-1\right)^{k-j}\left(1-x\right)^{n-j}
\]
obtained by expanding either (left-hand side) the $\left(1-x\right)^{n-k}$
term of the (right-hand side) $x^{k}=\left(x-1+1\right)^{k}$ in the
expression of the Bernstein polynomial. Replacing the variable $x$
by $x=X+\imath L_{E}-\frac{1}{2}$ and remarking that
\[
\mathbb{E}\left(1-x\right)^{n-j}=\mathbb{E}\left(1-X-\left(\imath L_{E}-\frac{1}{2}\right)\right)^{n-j}=\left(-1\right)^{n-j}E_{n-j}\left(X-1\right)
\]
with, by the cancellation principle, 
\[
E_{n-j}\left(X-1\right)+E_{n-j}\left(X\right)=2\left(X-1\right)^{n-j},
\]
we deduce
\begin{eqnarray*}
\sum_{j=0}^{k}\binom{k}{j}\left(-1\right)^{k-j}\left(1-x\right)^{n-j} & = & \sum_{j=0}^{k}\binom{k}{j}\left(-1\right)^{k-j}\left(-1\right)^{n-j}\left\{ -E_{n-j}\left(X\right)+2\left(X-1\right)^{n-j}\right\} \\
 & = & \left(-1\right)^{n+k+1}\sum_{j=0}^{k}\binom{k}{j}E_{n-j}\left(X\right)+2\sum_{j=0}^{k}\binom{k}{j}\left(-1\right)^{k-j}\left(1-X\right)^{n-j}
\end{eqnarray*}
this last sum being equal to $2X^{k}\left(1-X\right)^{n-k},$ hence
the result.
\end{proof}
We notice that the case $X=0$ is
\[
\sum_{j=0}^{n-k}\binom{n-k}{j}\left(-1\right)^{j}E_{j+k}=\left(-1\right)^{n+k+1}\sum_{j=0}^{k}\binom{k}{j}E_{n-j}+2\delta^{k}.
\]
It can be shown that 
\[
\left(-1\right)^{n+k+1}\sum_{j=0}^{k}\binom{k}{j}E_{n-j}=\sum_{j=0}^{k}\binom{k}{j}\left(-1\right)^{k-j}E_{n-j}
\]
as follows: using the moment representation \eqref{eq:En} the right-hand
side reads
\[
\mathbb{E}\left(\imath L_{E}-\frac{1}{2}\right)^{n-k}\left(1-\left(\imath L_{E}-\frac{1}{2}\right)\right)^{k}=\mathbb{E}\left(-\imath L_{E}-\frac{1}{2}\right)^{n-k}\left(1-\left(-\imath L_{E}-\frac{1}{2}\right)\right)^{k}
\]
by the symmetry of the hyperbolic secant distribution, and is thus
equal to 
\[
\left(-1\right)^{n-k}\mathbb{E}\left(\imath L_{E}+\frac{1}{2}\right)^{n-k}\left(\imath L_{E}+\frac{3}{2}\right)^{k}=\mathbb{E}f\left(\left(\imath L_{E}-\frac{1}{2}\right)+1\right)
\]
with $f\left(x\right)=\left(-1\right)^{n-k}x^{n-k}\left(x+1\right)^{k}$.
By the cancellation principle
\[
\mathbb{E}f\left(\left(\imath L_{E}-\frac{1}{2}\right)+1\right)+f\left(\imath L_{E}-\frac{1}{2}\right)=2f\left(0\right)=0
\]
so that the right-hand side if equal to $-\mathbb{E}f\left(\imath L_{E}-\frac{1}{2}\right)$
which coincides with the left-hand side; hence we recover Kim's identity.

\section{Multidimensional Volkenborn integral}

\subsection{Introduction}

In \cite{Kim2}, a multivariate version of the Volkenborn integral
is defined as 
\[
\int f\left(\mathbf{x}\right)d\mu_{0}\left(\mathbf{x}\right)=\int\dots\int f\left(x_{1},\dots,x_{k}\right)d\mu_{0}\left(x_{1}\right)\dots d\mu_{0}\left(x_{k}\right).
\]
In particular, it satisfies, with $\mathbf{y}\in\mathbb{R}^{k}$ and
the notation $\vert\mathbf{y}\vert=\sum_{i=1}^{k}y_{i},$
\[
\int_{\mathbb{Z}_{p}^{k}}e^{\left(x+\vert\mathbf{y}\vert\right)t}d\mu_{0}\left(\mathbf{y}\right)=\left(\frac{t}{e^{t}-1}\right)^{k}e^{xt}.
\]
This multivariate version of the Volkenborn integral can again be
expressed as an expectation over a simple random variable as shown
now.

\subsection{Moment representation and elementary properties}

The Bernoulli polynomials $B_{n}^{\left(k\right)}\left(x\vert\mathbf{a}\right)$
of order $k$ and degree $n$ with $x\in\mathbb{R}$ with parameter
$\mathbf{a}\in\mathbb{R}^{k}$, also called Nörlund polynomials, $\mathbf{}$are
defined by the generating function \cite[1.13.1]{Bateman2}
\[
\sum_{n=0}^{+\infty}B_{n}^{\left(k\right)}\left(x\vert\mathbf{a}\right)\frac{t^{n}}{n!}=e^{xt}\prod_{j=1}^{k}\left(\frac{a_{j}t}{e^{a_{j}t}-1}\right)
\]
and the corresponding Bernoulli numbers $B_{n}^{\left(k\right)}\left(\mathbf{a}\right)$
by
\[
B_{n}^{\left(k\right)}\left(\mathbf{a}\right)=B_{n}^{\left(k\right)}\left(0\vert\mathbf{a}\right)=\prod_{j=1}^{k}\left(\frac{a_{j}t}{e^{a_{j}t}-1}\right).
\]
 In particular, taking $a_{j}=1$ for all $j\in\left[1,k\right]$
and denoting
\[
B_{n}^{\left(k\right)}\left(x\right)=B_{n}^{\left(k\right)}\left(x\vert1,1,\dots,1\right)
\]
 we deduce
\[
\int_{\mathbb{Z}_{p}^{k}}e^{\left(x+\vert\mathbf{y}\vert\right)t}d\mu_{0}\left(\mathbf{y}\right)=\sum_{n=0}^{+\infty}B_{n}^{\left(k\right)}\left(x\right)\frac{t^{n}}{n!}.
\]

We provide a multidimensional extension of the moment representation
\eqref{eq:Bn(x)} as follows
\begin{thm}
The Bernoulli polynomials $B_{n}^{\left(k\right)}\left(x\vert\mathbf{a}\right)$
satisfy
\begin{equation}
B_{n}^{\left(k\right)}\left(x\vert\mathbf{a}\right)=\mathbb{E}\left(x+\sum_{j=1}^{k}a_{j}\left(\imath L_{B}^{\left(j\right)}-\frac{1}{2}\right)\right)^{n}\label{eq:Bnk(x)expectation}
\end{equation}
where the random variables $\left\{ L_{B}^{\left(j\right)}\right\} _{1\le j\le k}$
are independent and follow the logistic distribution \eqref{eq:logistic}. 

As a consequence, the Bernoulli numbers $B_{n}^{\left(k\right)}\left(\mathbf{a}\right)$
satisfy
\begin{equation}
B_{n}^{\left(k\right)}\left(\mathbf{a}\right)=\mathbb{E}\left(\sum_{j=1}^{k}a_{k}\left(\imath L_{B}^{\left(j\right)}-\frac{1}{2}\right)\right)^{n}\label{eq:Bnkexpectation}
\end{equation}
and the multivariate Volkenborn integral, with $\mathbf{x}\in\mathbb{R}^{k},$
\[
\int_{\mathbb{Z}_{p}^{k}}f\left(\vert\mathbf{x}\vert\right)d\mu_{0}\left(\mathbf{x}\right)=\mathbb{E}f\left(\sum_{j=1}^{k}\left(x_{j}+\imath L_{B}^{\left(j\right)}-\frac{1}{2}\right)\right).
\]
\end{thm}
\begin{proof}
Let us compute the generating function
\[
\sum_{n=0}^{+\infty}\mathbb{E}\left(x+\sum_{j=1}^{k}a_{j}\left(\imath L_{B}^{\left(j\right)}-\frac{1}{2}\right)\right)^{n}\frac{t^{n}}{n!}=\mathbb{E}\exp\left(tx+t\sum_{j=1}^{k}a_{j}\left(\imath L_{B}^{\left(j\right)}-\frac{1}{2}\right)\right)=e^{tx}\prod_{j=1}^{k}\mathbb{E}e^{ta_{j}\left(\imath L_{B}^{\left(j\right)}-\frac{1}{2}\right)}
\]
with, for a logistic distributed random variable $L_{j},$
\[
\mathbb{E}e^{ta_{j}\left(iL_{B}^{\left(j\right)}-\frac{1}{2}\right)}=\frac{a_{j}t}{e^{a_{j}t}-1}
\]
hence the result.
\end{proof}
The moment representation \eqref{eq:Bnk(x)expectation} allows to
recover easily some well-known results about the higher-order Bernoulli
polynomials.
\begin{prop}
The higher-order Bernoulli polynomials satisfy the identities

\begin{equation}
B_{n}^{\left(1\right)}\left(x\vert a\right)=a^{n}B_{n}\left(\frac{x}{a}\right),\label{eq:prop42}
\end{equation}
\begin{equation}
\sum_{l=0}^{n}\binom{n}{l}x^{l}B_{n-l}^{\left(k\right)}\left(y\vert\mathbf{a}\right)=B_{n}^{\left(k\right)}\left(x+y\vert\mathbf{a}\right)\label{eq:prop43}
\end{equation}
and
\begin{equation}
B_{2n+1}^{\left(k\right)}\left(\frac{a_{1}+\dots+a_{k}}{2}\vert\mathbf{a}\right)=0.\label{eq:prop44}
\end{equation}
\end{prop}
\begin{proof}
Identity \eqref{eq:prop42} is a direct consequence of the moment
representation \eqref{eq:Bnk(x)expectation}; identity \eqref{eq:prop43}
is obtained using a binomial expansion of \eqref{eq:Bnk(x)expectation}
and identity \eqref{eq:prop44} by computing 
\[
B_{2n+1}^{\left(k\right)}\left(\frac{a_{1}+\dots+a_{k}}{2}\vert\mathbf{a}\right)=\mathbb{E}\left(\imath\sum_{l=1}^{k}a_{l}L_{B}^{\left(l\right)}\right)^{2n+1}
\]
and using the fact that the logistic density \eqref{eq:logistic}
is an even function.
\end{proof}
We also deduce straightforwardly from a multinomial expansion of the
representations \eqref{eq:Bnk(x)expectation} and \eqref{eq:Bnkexpectation}
the following
\begin{prop}
The higher-order Bernoulli polynomials satisfy
\[
B_{n}^{\left(k\right)}\left(x_{1}+\dots+x_{k}\vert\mathbf{a}\right)=\sum_{i_{1}+\dots+i_{k}=n}\binom{n}{i_{1},\dots,i_{k}}B_{i_{1}}\left(x_{1}\vert a_{1}\right)\dots B_{i_{k}}\left(x_{k}\vert a_{k}\right)
\]
and the higher-order Bernoulli numbers 
\[
B_{n}^{\left(k\right)}\left(\mathbf{a}\right)=\sum_{i_{1}+\dots+i_{k}=n}\binom{n}{i_{1},\dots,i_{k}}B_{i_{1}}\left(a_{1}\right)\dots B_{i_{k}}\left(a_{k}\right)
\]

\end{prop}
These results extend Corollary 5 and Corollary 6 in \cite{Kim2} which
correspond to the case $\mathbf{a}=\left(1,\dots,1\right)$.

\subsection{Kim's identity for Nörlund polynomials}

In order to highlight the efficiency of the moment representation
\eqref{eq:Bnk(x)expectation}, we derive now an extension of Kim's
identity \eqref{eq:Kim1} to the case of Nörlund polynomials as follows.
\begin{thm}
For $p\in\mathbb{N}$ and $0\le k\le n,$ 
\[
\sum_{j=0}^{n-k}\binom{n-k}{j}B_{j+k}^{\left(p\right)}\left(x\right)=\sum_{j=0}^{k}\binom{k}{j}\left(-1\right)^{j}\left\{ B_{n-j}^{\left(p\right)}\left(x\right)+\left(n-j\right)B_{n-j-1}^{\left(p-1\right)}\left(x\right)\right\} .
\]
\end{thm}
\begin{proof}
We start from the identity
\begin{equation}
\sum_{j=0}^{n-k}\binom{n-k}{j}X^{j+k}=\sum_{j=0}^{k}\binom{k}{j}\left(-1\right)^{j}\left(1+X\right)^{n-j}\label{eq:identity}
\end{equation}
and replace $X$ by $x+\sum_{l=1}^{p}\left(\imath L_{B}^{\left(l\right)}-\frac{1}{2}\right)$
so that, from \ref{eq:Bnk(x)expectation}, the left-hand side reads
\[
\sum_{j=0}^{n-k}\binom{n-k}{j}B_{j+k}^{\left(p\right)}\left(x\right)
\]
while the right-hand side is
\[
\sum_{j=0}^{k}\binom{k}{j}\left(-1\right)^{j}B_{n-j}^{\left(p\right)}\left(x+1\right).
\]
Since $\mathbb{E}B_{n-j-1}^{\left(p\right)}\left(x+U\right)=B_{n-j-1}^{\left(p-1\right)}\left(x\right)$
with $U$ uniform on $\left[0,1\right],$ we deduce by the cancellation
principle
\[
\frac{B_{n-j}^{\left(p\right)}\left(x+1\right)-B_{n-j}^{\left(p\right)}\left(x\right)}{n-j}=B_{n-j-1}^{\left(p-1\right)}\left(x\right)
\]
 which yields the final result.
\end{proof}

\subsection{Kim's identity extended to multidimensional Euler polynomials}

We now provide a multidimensional version of the polynomial Kim identity
derived in Theorem \ref{thm:Kim_poly} as follows:
\begin{thm}
The multidimensional Euler polynomials satisfy the identity
\begin{eqnarray*}
\sum_{j=0}^{n-k}\binom{n-k}{j}\left(-1\right)^{j}E_{j+k}^{\left(p\right)}\left(x\right) & = & \left(-1\right)^{n+k+1}\sum_{j=0}^{k}\binom{k}{j}E_{n-j}^{\left(p\right)}\left(x\right)\\
 &  & +2\left(-1\right)^{n+k}\sum_{j=0}^{k}\binom{k}{j}E_{n-j}^{\left(p-1\right)}\left(x-1\right)
\end{eqnarray*}
\end{thm}
\begin{proof}
Starting again from identity \eqref{eq:identity}, we take $x=X+\sum_{l=1}^{p}\left(\imath L_{E}^{\left(l\right)}-\frac{1}{2}\right)$and
copute
\[
\left(1-x\right)^{n-j}=\left(1-X-\sum_{l=1}^{p}\left(\imath L_{E}^{\left(l\right)}-\frac{1}{2}\right)\right)^{n-j}=\left(-1\right)^{n-j}E_{n-j}^{\left(p\right)}\left(X-1\right).
\]
However, by the cancellation rule
\[
E_{n-j}^{\left(p\right)}\left(X-1\right)+E_{n-j}^{\left(p\right)}\left(X\right)=2E_{n-j}^{\left(p-1\right)}\left(X-1\right)
\]
and the result follows.
\end{proof}

\subsection{Raabe's and Nielsen's multiplication theorem for Nörlund polynomials}

Raabe's usual multiplication theorem 
\[
m^{1-n}B_{n}\left(mx\right)=\sum_{l=0}^{m-1}B_{n}\left(x+\frac{l}{m}\right)
\]
and 
\[
m^{-n}E_{n}\left(mx\right)=\sum_{l=0}^{m-1}\left(-1\right)^{l}E_{n}\left(x+\frac{l}{m}\right),\,\, m\,\,\text{odd}
\]
and Nielsen's multiplication theorem
\[
m^{-n}E_{n}\left(mx\right)=-\frac{2}{n+1}\sum_{l=0}^{m-1}\left(-1\right)^{l}B_{n+1}\left(x+\frac{l}{m}\right),\,\, m\,\,\text{even}
\]
are an interesting feature of the Bernoulli polynomials since, as
noted by Nielsen, \cite[p. 54]{Nielsen}
\begin{quotation}
It is very curious, it seems to me, that there exist polynomials,
with arbitrary degree, that satisfy equations of the above form. However,
it is easy to prove that, up to an arbitrary constant factor, the
$B_{n}\left(x\right)$ and $E_{n}\left(x\right)$ are the only polynomials
that satisfy the mentioned property.
\end{quotation}
Using the moment representation and basic results from probability
theory, we propose the following extension of Raabe's celebrated multiplication
theorem to the multivariate case.
\begin{thm}
If $m\in\mathbb{N},$
\begin{equation}
m^{k-n}B_{n}^{\left(k\right)}\left(mx\vert\mathbf{a}\right)=\sum_{l_{1},\dots,l_{k}=0}^{m-1}B_{n}^{\left(k\right)}\left(x+\frac{1}{m}\sum_{i=1}^{k}a_{i}l_{i}\vert\mathbf{a}\right)\label{eq:Raabe}
\end{equation}
and if moreover $m$ is odd, 
\begin{equation}
m^{-n}E_{n}^{\left(k\right)}\left(mx\vert\mathbf{a}\right)=\sum_{l_{1},\dots,l_{k}=0}^{m-1}\left(-1\right)^{l_{1}+\dots+l_{k}}E_{n}^{\left(k\right)}\left(x+\frac{1}{m}\sum_{i=1}^{k}a_{i}l_{i}\vert\mathbf{a}\right).\label{eq:Raabe2}
\end{equation}
\end{thm}
\begin{proof}
Let us denote $\left\{ \tilde{U}^{\left(i\right)}\right\} _{1\le i\le k}$
a set of $k$ discrete random variables independent and uniformly
distributed in the set $\left\{ 0,\dots,m-1\right\} $ and $\left\{ U_{B}^{\left(i\right)}\right\} _{1\le i\le k}$
a set of $k$ continuous random variables independent and uniformly
distributed on the interval $\left[0,1\right].$ For the Bernoulli
case, we have 
\begin{gather*}
\frac{1}{m^{k}}\sum_{l_{1},\dots,l_{k}=0}^{m-1}B_{n}^{\left(k\right)}\left(x+\frac{1}{m}\sum_{i=1}^{k}a_{i}l_{i}\vert\mathbf{a}\right)=\mathbb{E}\left(x+\sum_{i=1}^{k}a_{i}\left(\imath L_{B}^{\left(i\right)}-\frac{1}{2}\right)+\frac{1}{m}\sum_{i=1}^{k}a_{i}\tilde{U}^{\left(i\right)}\right)^{n}\\
=\frac{1}{m^{n}}\mathbb{E}\left(mx+m\sum_{i=1}^{k}a_{i}\left(\imath L_{B}^{\left(i\right)}-\frac{1}{2}\right)+\sum_{i=1}^{k}a_{i}\tilde{U}^{\left(i\right)}\right)^{n}\\
=\frac{1}{m^{n}}\mathbb{E}\left(mx+m\sum_{i=1}^{k}a_{i}\left(\imath L_{B}^{\left(i\right)}-\frac{1}{2}\right)+\sum_{i=1}^{k}a_{i}\tilde{U}^{\left(i\right)}+\sum_{i=1}^{k}a_{i}\left(\imath\tilde{L}_{B}^{\left(i\right)}-\frac{1}{2}\right)+\sum_{i=1}^{k}a_{i}U_{B}^{\left(i\right)}\right)^{n}
\end{gather*}
Now we use the fact that $\tilde{U}^{\left(i\right)}+U_{B}^{\left(i\right)}$
has the same distribution as $mU_{B}^{\left(i\right)}$ so that we
obtain 
\[
\frac{1}{m^{n}}\mathbb{E}\left(mx+m\sum_{i=1}^{k}a_{i}\left(\imath L_{B}^{\left(i\right)}-\frac{1}{2}\right)+\sum_{i=1}^{k}a_{i}\left(\imath\tilde{L}_{B}^{\left(i\right)}-\frac{1}{2}\right)+m\sum_{i=1}^{k}a_{i}U_{B}^{\left(i\right)}\right)^{n}
\]
and applying the cancellation principle, we deduce
\[
\frac{1}{m^{n}}\mathbb{E}\left(mx+\sum_{i=1}^{k}a_{i}\left(\imath\tilde{L}_{B}^{\left(i\right)}-\frac{1}{2}\right)\right)^{n}=\frac{1}{m^{n}}B_{n}^{\left(k\right)}\left(mx\vert\mathbf{a}\right).
\]

For the Euler case, we need to use a signed measure (and then depart
temporarily from the probabilistic context) defining the set  $\left\{ \hat{U}^{\left(i\right)}\right\} _{1\le i\le k}$
of $k$ discrete variables independent such as each $\hat{U}^{\left(i\right)}$
takes values in $\left\{ 0,\dots,k,\dots,m-1\right\} $ with a weight
$\left(-1\right)^{k}.$ Then
\begin{gather*}
\sum_{l_{1},\dots,l_{k}=0}^{m-1}\left(-1\right)^{l_{1}+\dots+l_{k}}E_{n}\left(x+\frac{1}{m}\sum_{i=1}^{k}a_{i}l_{i}\vert\mathbf{a}\right)=\mathbb{E}\left(x+\sum_{i=1}^{k}a_{i}\left(\imath L_{E}^{\left(i\right)}-\frac{1}{2}\right)+\frac{1}{m}\sum_{i=1}^{k}a_{i}\hat{U}^{\left(i\right)}\right)^{n}\\
=\frac{1}{m^{n}}\mathbb{E}\left(mx+m\sum_{i=1}^{k}a_{i}\left(\imath L_{E}^{\left(i\right)}-\frac{1}{2}\right)+\sum_{i=1}^{k}a_{i}\hat{U}^{\left(i\right)}\right)^{n}\\
=\frac{1}{m^{n}}\mathbb{E}\left(mx+m\sum_{i=1}^{k}a_{i}\left(\imath L_{E}^{\left(i\right)}-\frac{1}{2}\right)+\sum_{i=1}^{k}a_{i}\hat{U}^{\left(i\right)}+\sum_{i=1}^{k}a_{i}\left(\imath\tilde{L}_{E}^{\left(i\right)}-\frac{1}{2}\right)+\sum_{i=1}^{k}a_{i}U_{E}^{\left(i\right)}\right)^{n}
\end{gather*}
where now $\left\{ U_{E}^{\left(i\right)}\right\} $ are independent
Rademacher random variables and, since $m$ is odd, each $U_{E}^{\left(i\right)}+\hat{U}^{\left(i\right)}$
has the same distribution as $mU_{E}^{\left(i\right)}$ so that we
obtain 
\[
\frac{1}{m^{n}}\mathbb{E}\left(mx+m\sum_{i=1}^{k}a_{i}\left(\imath L_{E}^{\left(i\right)}-\frac{1}{2}\right)+\sum_{i=1}^{k}a_{i}\left(\imath\tilde{L}_{E}^{\left(i\right)}-\frac{1}{2}\right)+m\sum_{i=1}^{k}a_{i}U{}_{E}^{\left(i\right)}\right)^{n}
\]
and from the cancellation principle, we deduce the result.
\end{proof}
Raabe's identity \eqref{eq:Raabe} and \eqref{eq:Raabe2} are in fact
given without proof in \cite[eq. (1.6) and (1.7)]{Carlitz}. The case
for $m$ even is not provided, so we prove now
\begin{thm}
With $n\in\mathbb{N}$ and $m$ even,
\[
m^{k-n}\left(-\frac{1}{2}\right)^{k}\frac{n!}{\left(n-k\right)!}\left(\prod_{i=1}^{k}a_{i}\right)E_{n-k}^{\left(k\right)}\left(mx\right)=\sum_{l_{1},\dots,l_{k}=0}^{m-1}\left(-1\right)^{l_{1}+\dots+l_{k}}B_{n}^{\left(k\right)}\left(x+\frac{1}{m}\sum_{i=1}^{k}a_{i}l_{i}\vert\mathbf{a}\right).
\]
\end{thm}
\begin{proof}
Let us define a variable $W=\left\{ 0,\dots,m-1\right\} $ with weights
$\left(-1\right)^{l}.$ Then the right-hand side reads, with $B_{i}=\imath L_{B}^{\left(i\right)}-\frac{1}{2}$
and $E_{i}=\imath L_{E}^{\left(i\right)}-\frac{1}{2}$, 
\begin{eqnarray*}
\mathbb{E}\left(x+\sum_{i=1}^{k}a_{i}\left(\frac{W_{i}}{m}+B_{i}\right)\right)^{n} & = & m^{-n}\mathbb{E}\left(x+\sum_{i=1}^{k}a_{i}\left(W_{i}+mB_{i}+E_{i}+U_{E}^{\left(i\right)}\right)\right)^{n}\\
 & = & m^{-n}E_{n}\left(x+\sum_{i=1}^{k}a_{i}\left(W_{i}+mB_{i}+U_{E}^{\left(i\right)}\right)\right)
\end{eqnarray*}
but it can be checked that each $W_{i}+U_{E}^{\left(i\right)}$ takes
values $0$ and $m$ with respective weights $\frac{1}{2}$ and $-\frac{1}{2}$
so that we obtain
\[
-\frac{1}{2}m^{-n}\left(E_{n}\left(x+m\sum_{i=1}^{k}a_{i}B_{i}+\sum_{i=1}^{k-1}\left(W_{i}+U_{E}^{\left(i\right)}\right)+ma_{k}\right)-E_{n}\left(x+m\sum_{i=1}^{k}a_{i}B_{i}+\sum_{i=1}^{k-1}\left(W_{i}+U_{E}^{\left(i\right)}\right)\right)\right)
\]
which coincides with
\[
-\frac{1}{2}m^{-n}E_{n}\left(x+m\sum_{i=1}^{k}a_{i}B_{i}+\sum_{i=1}^{k-1}\left(W_{i}+U_{E}^{\left(i\right)}\right)+ma_{k}U_{B}^{\left(k\right)}\right)=-\frac{1}{2}m^{-n}\left(ma_{k}\right)nE_{n-1}\left(x+\sum_{i=1}^{k-1}a_{i}\left(W_{i}+mB_{i}+U_{E}^{\left(i\right)}\right)\right)
\]
since $ma_{k}B_{k}$ and $ma_{k}U_{B}^{\left(k\right)}$ cancel out.
We are now back, up to a factor, to the same quantity as before except
that $n$ is replaced by $n-1$ and $k$ by $k-1$, hence the result.\end{proof}


\begin{thebibliography}{References}
\bibitem{Kim1}T. Kim, J. Choi, Y.-H. Kim and C. S. Ryoo, On the Fermionic
p-adic integral representation of Bernstein polynomials associated
with Euler numbers and polynomials, Journal of Inequalities and Applications,
Volume 2010 (2010), Article ID 864247

\bibitem{Kim2}M.-S. Kim and J.-W. Son, On a multidimensional Volkenborn
integral and higher order Bernoulli numbers, Bull. Austral. Math.
Soc., 65, 59-71, 2002

\bibitem{Bateman2}H. Bateman, Higher Transcendental Functions, vol.1,
Krieger Pub Co, 1981

\bibitem{Volkenborn}A. Volkenborn: Ein p-adisches Integral und seine
Anwendungen I and II, Zeitschrift manuscripta mathematica, 1972, 7-4,
342-373 and 1974, 12-1, 17-46, Springer-Verlag 

\bibitem{Robert}A.M. Robert, A Course in $p-$adic Analysis, Graduate
Texts in Mathematics 198, 2000, Springer

\bibitem{Raabe}J.L. Raabe, Zurückführung einiger Summen und bestimmten
Integrale auf die Jacob Bernoullische Function, Journal für die reine
und angewandte Mathematik, 1851, 42, 348-376

\bibitem{Carlitz}L. Carlitz, The Multiplication Formulas for the
Bernoulli and Euler Polynomials, Mathematics Magazine , 27-2, 59-64,
1953

\bibitem{Nielsen}N. Nielsen, Traité élémentaire des nombres de Bernoulli,
Gauthier-Villars, 1923\end{thebibliography}
\end{document}